\documentclass[a4paper, 12pt]{amsart}
\usepackage{amsfonts, amsmath, amssymb, amsthm}  % различная математика
\usepackage[english]{babel}
\usepackage[utf8]{inputenc}
\usepackage[OT1]{fontenc}
\usepackage{bm}

\usepackage[margin=2.2cm]{geometry}

\usepackage{textcomp, cmap, comment}	
\usepackage{graphicx, wrapfig}
\usepackage{epigraph, xcolor, hyperref}
\usepackage{array,tabularx,tabulary,booktabs}
\usepackage{euscript, mathrsfs}

\newtheorem*{thm*}{Theorem}
\newcommand{\ff}{{\mathcal F}}
\newcommand{\aaa}{{\mathcal A}}

\newtheorem*{cla*}{Claim}
\newcommand{\bb}{{\mathcal B}}
\newcommand{\G}{{\mathcal G}}

\newtheorem{thm}{Theorem}
\newtheorem{gypo}{Conjecture}
\newtheorem{prb}{Problem}

\newtheorem{lem}[thm]{Lemma}

\newtheorem{cor}[thm]{Corollary}
\date{}

\newtheorem{prop}[thm]{Proposition}
\newtheorem{obs}[thm]{Observation}

\newtheorem{defn}{Definition}

\title{Diversity}
\author{Peter Frankl}\address{R\'enyi Institute, Budapest, Hungary; Email: {\tt peter.frankl@gmail.com}}

\author{Andrey Kupavskii}
\address{University of Birmingham and
Moscow Institute of Physics and Technology; Email: {\tt kupavskii@ya.ru}.} \thanks{The research of the second author was supported by the Advanced Postdoc.Mobility grant no. P300P2\_177839 of the Swiss National Science Foundation.}
\date{}

\begin{document}
\maketitle
\begin{abstract}
 Given a family $\ff\subset 2^{[n]}$, its {\it diversity} is the number of sets not containing an element with the highest degree. The concept of diversity has proven to be very useful in the context of $k$-uniform intersecting families. In this paper, we  study (different notions of) diversity in the context of other extremal set theory problems. One of the main results of the paper is a sharp stability result for cross-intersecting families in terms of diversity and, slightly more generally, sharp stability for the Kruskal--Katona theorem.
\end{abstract}
\section{Introduction and main results}
Given a family $\ff\subset 2^{[n]}$, we define its {\it diversity} $\gamma(\ff)$ to be $\gamma(\ff):=\min_{i\in [n]}|\{F\in \ff: i\notin F\}|$. The {\it maximum degree} $\Delta(\ff)$ of $\ff$ is $\Delta(\ff):=\max_{i\in [n]}|\{F\in \ff: i\in F\}|$. That is, $\gamma(\ff)+\Delta(\ff)=|\ff|$.
We say that a family $\ff$ is {\it intersecting} if $F_1\cap F_2\ne \emptyset$ for any $F_1,F_2\in \ff$. To the best of our knowledge, the problem of maximizing diversity first appeared explicitly in the following question of Katona: what is the maximum possible diversity of an intersecting family $\ff\subset {[n]\choose k}$? Lemons and Palmer \cite{LP} showed that $\gamma(\ff)\le {n-3\choose k-2}$ for $n\ge 6k^3$ (they used the term ``unbalance'' instead of diversity). Later, this bound on $n$ was improved to $n\ge 6k^2$ by the first author \cite{Fra6} and to $n\ge Ck$ with some absolute constant $C$ by the second author \cite{Kup21}.

Diversity is also closely related to the concept of influence of a Boolean function, which we discuss later in the paper. Katona's question was motivated by the work of Dinur and Friedgut \cite{DF}, who studied  approximations of intersecting families by juntas using the tools coming from the Analysis of Boolean functions. In particular, using these tools, they showed that $\gamma(\ff)=O({n-2\choose k-2})$ for $n>Ck$. We note that this result in itself was not new: the first author showed $\gamma(\ff)\le 3{n-2\choose k-2}$ for {\it any} $n>2k>0$ in \cite{Fra1}. However, the more general approximation by juntas result of \cite{DF} was an important tool in the proof in \cite{Kup21}.

Diversity proved to be a very useful concept in the study of intersecting families. It was recently applied to different problems in \cite{Feg, FK5, IK, Kup22}. One of the tools that was particularly fruitful is a recent theorem that allows us to bound the size of $k$-uniform intersecting families based on their diversity, proven by the second author and Zakharov in \cite{KZ}. This theorem is a stronger version of an old result due the first author \cite{Fra1}. Both results provide strong and in many cases sharp stability versions of the Erd\H os-Ko-Rado \cite{EKR} theorem, which states that any intersecting family $\ff\subset {[n]\choose k}$ has size at most ${n-1\choose k-1}$, provided $n\ge 2k$. An even stronger version of these theorems, which is sharp in {\it all} cases, was proven in \cite{Kup22}.

For a real $x\ge k$, let us put ${x\choose k}:=\frac {x(x-1)\ldots(x-k+1)}{k!}$, and put ${x\choose k}:=0$ for $x<k$.
\begin{thm}[\cite{KZ}]\label{thmkz} Let $n>2k>0$ and $\ff\subset {[n]\choose k}$ be an intersecting family. Then, if $\gamma(\ff)\ge {n-u-1\choose n-k-1}$ for some real $3\le u\le k$, then \begin{equation}\label{eqkz}|\ff|\le {n-1\choose k-1}+{n-u-1\choose n-k-1}-{n-u-1\choose k-1}.\end{equation}
\end{thm}
The bound from Theorem \ref{thmkz} is sharp for integer $u$, as witnessed by the family $\mathcal L_{u,u}(n,k)$, defined after Theorem~\ref{thmmain1}.

In this paper, we conduct a more systematic study of diversity in the context of some classic extremal set theory problems.  For an integer $t\ge1$, we say that the families $\ff_1,\ldots, \ff_l\subset 2^{[n]}$ are {\it cross $t$-intersecting} if $|F_1\cap \ldots\cap F_l|\ge t$ for any $F_1\in\ff_1,\ldots,  F_l\in \ff_l$. If $t=1$ then we use the term {\it cross-intersecting}. The following theorem is a generalization of the Erd\H os--Ko--Rado theorem.
\begin{thm}[\cite{Day3}]\label{thmekrcross} Let $a,b$ be positive integers and $n\ge a+b$. Suppose that $\aaa\subset {[n]\choose a}$ and $\bb\subset {[n]\choose b}$ are cross-intersecting and  $|\aaa|\ge {n-1\choose a-1}$. Then $|\bb|\le {n-1\choose b-1}$.
\end{thm}
If one puts $a=b=k$ and $\aaa=\bb$, then one retrieves the EKR theorem. Again, the bound in Theorem~\ref{thmekrcross} is sharp: take $\aaa$ and $\bb$ to be the families of all sets containing $1$ of sizes $a$ and $b$, respectively.

The key result of this paper is the following stability for Theorem~\ref{thmekrcross} in terms of diversity.
%\textcolor{red}{It's better to have it with equality, and to exclude the case $u=v=3$.}
\begin{thm}\label{thmmain1} Let $a,b$ be positive integers and $n\ge a+b$. Suppose that $\aaa\subset {[n]\choose a}$ and $\bb\subset {[n]\choose b}$ are cross-intersecting.
Suppose that
\begin{align*}|\aaa|\ \ge&\ {n-1\choose a-1}-{n-v-1\choose a-1} +{n-u-1\choose n-a-1}\ \ \ {\text and }\\
|\bb|\ \ge&\ {n-1\choose b-1}-{n-u-1\choose b-1}+{n-v-1\choose n-b-1}\end{align*}
for some real $3\le u\le a$, $3\le v\le b$ and, moreover, that at least one of the displayed inequalities is strict. Then
$$\gamma(\aaa)< {n-u-1\choose n-a-1}\ \ \ \text{and}  \ \ \ \gamma(\bb)< {n-v-1\choose n-b-1},$$
moreover, both families have the same (unique) element of the largest degree.
\end{thm}
The theorem above is sharp for each integer  $u\in [3,a],v\in [3,b]$ for $\mathcal A=\mathcal L_{u,v}(n,a)$, $\mathcal B=\mathcal L_{v,u}(n,b)$, where $$\mathcal L_{u,v}(n,k):=\Big\{L\in {[n]\choose k}: 1\in L, L\cap [2,v+1]\ne \emptyset\Big\}\cup \Big\{L\in {[n]\choose k}:  [2,u+1]\subset L\Big\}.$$
Putting $a=b, u=v, \aaa=\bb$, one retrieves Theorem~\ref{thmkz}.
The proof of Theorem~\ref{thmmain1} is given in Section~\ref{sec5}. Note also that the statement does not hold for $u=v=3$ with non-strict inequality on both sizes of $\aaa$ and $\bb$: the families $\aaa = \mathcal L_{2,2}(n,a)$, $\bb = \mathcal L_{2,2}(n,b)$, satisfy the inequality on the size (note that $|\mathcal L_{2,2}(n,k)| = |\mathcal L_{3,3}(n,k)|$), but have much larger diversity.

Let us deduce the following numerical corollary of the theorem above. For simplicity, we only deal with $u=v=x$ case and with $n\ge 2\max \{a,b\}$.

\begin{cor} Fix $a,b>0$ and $n\ge 2\max \{a,b\}$. If $\aaa\subset {[n]\choose a}$ and $\bb\subset {[n]\choose b}$ are cross-intersecting and, for some $n-\min\{a,b\}-1\le x\le n-4$, $|\aaa|> {n-1\choose a-1}-(1-\frac {a^2}{(n-a)^2}){x\choose a-1}$, $|\bb|> {n-1\choose b-1}-(1-\frac {b^2}{(n-b)^2}){x\choose b-1}$, then $\gamma(\aaa)< {x\choose n-a-1}$ and $\gamma(\bb)< {x\choose n-b-1}$.
\end{cor}
\begin{proof} For $n\ge 2k$ and $n-k-1\le x\le n-4$, the polynomial $p(x):={x\choose n-k-1}/{x\choose k-1}$ is increasing (and for smaller $x\ge k-1$, the ratio is $0$), and thus $p(x)\le p(n-4)={n-4\choose k-3}/{n-4\choose k-1} = \frac{(k-1)(k-2)}{(n-k-1)(n-k-2)}\le \frac {k^2}{(n-k)^2}.$ Thus, because of our assumption on $|\aaa|,\ |\bb|$, we have $|\aaa|\ge {n-1\choose a-1}-{x\choose a-1}+{x\choose n-a-1}$ and similar inequality for $\bb$. Theorem~\ref{thmmain1} then implies that the diversities satisfy the desired inequalities.
\end{proof}

Statements concerning sizes of cross-intersecting families are often directly related to the famous Kruskal--Katona theorem \cite{Kr,Ka}, which is one of the cornerstones of extremal set theory. Let us give the following definition.

\begin{defn}
  For a family $\ff$, denote by $\partial^{l}\ff$ the collection of $l$-element sets that are contained in at least one set from $\ff$. We call $\partial^l\ff$ the {\rm $l$-shadow} of $\ff$. If $\ff\subset {[n]\choose k}$ then we denote by $\partial \ff$ the {\rm immediate shadow} $\partial^{k-1}\ff$.
\end{defn} We state the Kruskal-Katona Theorem in the form proposed by Lov\'asz.
\begin{thm}\label{thmkk}
  If $\ff$ is a family of $k$-element sets, such that for some real $x\ge k$ we have $|\ff|={x\choose k}$, then $|\partial \ff|\ge {x\choose k-1}$.
\end{thm}

The other main result of this paper is a strengthening of Theorem~\ref{thmmain1} in the Kruskal--Katona-type form.  For a family $\ff$ and integer $n$, define the {\it Kruskal-Katona diversity} $\gamma_{KK}(\ff,n):=\min_{X, |X|=n}\big|\ff\setminus {X\choose k}\big|$.  In Section~\ref{sec42}, we will prove the following theorem.
\begin{thm}\label{thmstabkk}
  Let $n>k>0$ and assume that $\ff$ is a family of $k$-element sets, such that $|\ff|\ge {n\choose k}-{y\choose n-k}+{x\choose k-1}$ and $\gamma_{KK}(\ff,n)\ge {x\choose k-1}$ for some real numbers $k-1\le x\le n-3$, $n-k\le y\le n-3$. Then $\partial(\ff) \ge {n\choose k-1}-{y\choose n-k+1}+{x\choose k-2}$.
\end{thm}

Again, the theorem is sharp for integer $y,x$, as is witnessed by the example
$$\mathcal{KK}_{x,y}:={[n]\choose k}\setminus \big\{S\subset [n]: [y+1,n]\subset S\big\}\cup \Big\{S\cup \{n+1\}: S\in {[x]\choose k-1}\Big\}.$$
Moreover, one sees that $|\mathcal {KK}_{3,3}|=|\mathcal {KK}_{2,2}|$, and $\mathcal {KK}_{2,2}$ has the same ``distance'' from three different subsets of $[n+1]:$ $[n+1]\setminus \{i\}$, where $i\in [n-1,n+1]$. This means that the range of $x$, $y$ in the theorem is in some sense best possible.
From this theorem, we can derive the following numerical corollary.

\begin{cor}\label{corstabkk} If $|\ff| = {n\choose k}$ and $\gamma_{KK}(\ff,n)\ge t:={x\choose k-1}$ for some $t\le\{{n-3\choose k-3},{n-3\choose k-1}\},$ then $|\partial \ff|\ge {n\choose k-1}+\max\{\frac {1}{n-k+1},\frac 1{k-1}\} {x\choose k-2}$.
\end{cor}
\begin{proof}
Find $y\le n-3$ such that $|\ff|={n\choose k}-{y\choose n-k}+{x\choose k-1}$. It is easy to see that $n-k\le y\le n-2$ and $k-1\le x\le n-2$. Then, by Theorem~\ref{thmstabkk}, $|\partial \ff|\ge {n\choose k-1}-{y\choose n-k+1}+{x\choose k-2}\ge \min\{\frac 1{n-k+1},\frac 1{k-1}\}{x\choose k-2}$ since ${y\choose n-k+1} = \frac{y-n+k}{n-k+1}{y\choose n-k} = \frac {y-n+k}{n-k+1}{x\choose k-1} = \frac {(y-n+k)(x-(k-2))}{(n-k+1)(k-1)}{x\choose k-2}\le \frac {(k-2)(n-k)}{(n-k+1)(k-1)}{x\choose k-2}$.
\end{proof}

For two sets $A,B$, let $A\oplus B$ stand for the symmetric difference of $A$ and $B$. While we were preparing this paper, Keevash and Long \cite{KLo} published the following theorem.

\begin{thm}[\cite{KLo}]\label{thmklo} If $\delta_0\in(0,1)$ and $\aaa\subset {[n]\choose k}$ with $|\aaa| = {x\choose k}$ and $\partial A\le (1+\frac cx){x\choose k-1}$, with $c:=10^{-9}\delta_0$, then $|\aaa\oplus {S\choose k}|\le\delta_0{|S|-1\choose k-1}$ for some $S\subset [n]$.
\end{thm}

Theorems~\ref{thmstabkk} and Theorem~\ref{thmklo} deal with stability in similar scenarios, as one sees in Corollary~\ref{corstabkk}, however, none implies the other. Where applicable, Theorem~\ref{thmstabkk} gives better (sharp) bounds, and its range is at least big for $k$ comparable to $n$ and especially for $k$ that are close to $n$. However, for $n$ large w.r.t. $k$, it works for the families that are closer to ${S\choose k}$ for some set $S$ than Theorem~\ref{thmklo}.

One advantage of Theorem~\ref{thmstabkk} is that it has a much less technical proof (especially in terms of computations) than that of Theorem~\ref{thmklo}. We could have made our conclusions work in a wider range, but then it would have made the statements and the proofs uglier, which, in view of Theorem~\ref{thmklo}, we decided to avoid. The two proofs share one important ingredient: to deal with families with very large diversity, we continuously transform the family into one that has smaller, but not too small, diversity, using combinatorial operations, in particular, the so-called {\it Daykin shifts} (cf. Section~\ref{sec3}). A similar approach appeared earlier in the paper by Zakharov and the second author \cite{KZ} to obtain Theorem~\ref{thmkz}, and our analysis resembles that of \cite{KZ}. However, in \cite{KZ}, (normal) shifts were used, together with some other exchange operation.
For other stability results on Kruskal--Katona theorem, cf. \cite{Kee, ODW}.

Let $r\ge 2,t\ge 1$ be integers. We say that a family $\ff\subset 2^{[n]}$ is {\it $r$-wise $t$-intersecting} if $|F_1\cap \ldots F_r|\ge t$ for any $F_1,\ldots, F_r\in \ff$. If $r=2$, then we sometimes omit ``$2$-wise'' for brevity. Similarly, we shorten ``$r$-wise $1$-intersecting'' to  ``$r$-wise intersecting''.

\begin{thm}\label{thmmain2}
  Fix integers $r\ge 3$ and $t\ge 1$. Then for $n>\max\{15,2(r+t)\}k$ and an $r$-wise $t$-intersecting family $\ff\subset {[n]\choose k}$ we have \begin{equation}\label{eqrwise}\gamma(\ff)\le {n-r-t\choose k-r-t+1}.\end{equation}
\end{thm}
The theorem is tight, as witnessed by the example $\{F\in{[n]\choose k}: F\cap[r+t]\ge r+t-1\}$.
We give the proof of Theorem~\ref{thmmain2} in Section~\ref{sec4}. We note the case $r=2$ seems to be the hardest. For $r=2,t=1$, this corresponds to the ``diversity of intersecting families'' question of Katona, discussed in the beginning of the introduction, and thus the same statement is true for $n>Ck$ with some absolute constant $C$. For $r=2$ and $t\ge 2$, it seems that it is possible to resolve by generalizing techniques of the second author from \cite{Kup21}. The case of shifted families is easy, see Corollary~\ref{cor13}.

For a family $\ff\subset 2^{[n]}$, let the {\it matching number} $\nu(\ff)$ stand for the largest $s$, such that there are $F_1,\ldots, F_s\in \ff$, satisfying $F_i\cap F_j=\emptyset $ for $1\le i<j\le s$. In particular, $\nu(\ff)=1$ if and only if $\ff$ is intersecting. The {\it $s$-diversity} $\gamma_s(\ff)$ of $\ff\subset 2^{[n]}$ is defined as $$\gamma_s(\ff):=\min\Big\{\big|\{F\in \ff: F\cap R= \emptyset\big|\ :\ R\in {[n]\choose s}\Big\}.$$
This notion seems to be the correct generalization of the notion of diversity to families $\ff$ with  $\nu(\ff)=s$. For $s=1$, we get back the usual notion of diversity. In Section~\ref{sec6}, we determine the maximum value of $\gamma_s(\ff)$ for $\ff\subset {[n]\choose k}$ for all $n>n_0(s,k)$ (and completely resolve the $k=2$ case).

In the next section, we recall the definition of shifting and discuss some of its properties. Lemma~\ref{lemcorrshift} (``shifted families are positively correlated''), although simple, appears to be a new and useful tool in the study of shifted families. In Section~\ref{sec3}, we discuss the Kruskal--Katona theorem and some other, more advanced, tools.

Most of the results that we present are for uniform families. However, we prove some results in the non-uniform case as well. For diversity of $r$-wise $t$-intersecting families in $2^{[n]}$, see Corollary~\ref{cor14} (shifted case) and Section~\ref{sec7} (general case).

%\textcolor{red}{Structure of the paper.}

\section{Shifting and initial families}\label{sec2}
For a set $F\subset [n]$, a family $\ff\subset 2^{[n]}$ and $1\le i<j\le n$, one defines the $i\leftarrow j$ shift $S_{ij}$ by
\begin{align*} S_{ij}(F):=&\begin{cases}
                  (F-\{j\})\cup \{i\} & \mbox{if } F\cap \{i,j\}=j, \\
                  F & \mbox{otherwise};
                \end{cases}\\
                S_{ij}(\ff):=& \ \big\{S_{ij}(F)\ :\ F\in\ff\big\}\cup \big\{F\ :\ F,S_{ij}(F)\in \ff\big\}.
\end{align*}
It is easy to see that $|S_{ij}(\ff)|=|\ff|$.  Since $|S_{ij}(F)|=|F|$, should $\ff$ be $k$-uniform, then so is $S_{ij}(\ff)$.
\begin{prop}[\cite{EKR,Fra3}]\begin{itemize}\item[(i)] If $\ff$ is $r$-wise $t$-intersecting then so is $S_{ij}(\ff)$;
\item[(ii)] $\nu\big(S_{ij}(\ff)\big)\le \nu(\ff)$.
\end{itemize}
\end{prop}

Given $\ff\subset 2^{[n]}$ and two disjoint sets $X,Y\subset [n]$, we define
$$\ff(X\bar Y):=\big\{F\setminus X: F\in\ff,F\cap (X\cup Y)=X\big\}.$$
For $X=\{i\}$, $Y=\{j\}$ we write $\ff(i\bar j)$ instead of $\ff\big(\{i\}\overline{\{j\}}\big)$. We define $\ff(X)$, $\ff(\bar Y)$ similarly.

The following proposition is obvious from the definitions of the $i\leftarrow j$ shift and diversity.

\begin{prop}\label{prop55} The following properties hold.
\begin{itemize}\item[(i)] For $x\in[n]\setminus \{i,j\}$, $d_{S_{ij}(\ff)}(x) = d_{\ff}(x)$;
\item[(ii)] $d_{S_{ij}(\ff)}(j) \le d_{S_{ij}(\ff)}(i)=d_{\ff}(i)+\big|\ff(\bar i j)\setminus \ff(i\bar j)\big|= d_{\ff}(j)+\big|\ff(i\bar j)\setminus \ff(\bar i j)\big|$;
\item[(iii)] $\gamma(S_{ij}(\ff))\ge \gamma(\ff)-\min\big\{\big|\ff(\bar i j)\setminus \ff(i\bar j)\big|, \big|\ff(i\bar j)\setminus \ff(\bar i j)\big|\big\}\ge \gamma(\ff)-\big|\ff(i\bar j)\oplus \ff(\bar i j)\big|/2$.
\end{itemize}
\end{prop}

If $\{x_1,\ldots, x_k\}\subset [n]$ and $x_1<\ldots <x_k$ then we also denote it by $(x_1,\ldots, x_k)$. One defines the {\it shifting partial order} $\prec_{s}$ by
$$(x_1,\ldots, x_k)\prec_s (y_1,\ldots, y_k) \text{ iff } x_i\le y_i\text{ for all }1\le i\le k.$$
This can be naturally extended to $F,G\in {[n]\choose k}$ by arranging the elements of $F$ and $G$ in increasing order.
\begin{defn}
  A family $\ff\subset 2^{[n]}$ is called {\rm shifted} if $G\prec_s F\in \ff$ implies $G\in \ff$.
\end{defn}

Note that repeated applications of the $i\leftarrow j$ shifts, $1\le i<j\le n$ eventually produces a shifted family. The following statement should be obvious.
\begin{prop}\label{prop24} Suppose that $\ff\subset 2^{[n]}$ is shifted. Then (i) and (ii) hold.
\begin{itemize}
  \item[(i)] $d_{\ff}(1)\ge \ldots \ge d_{\ff}(n)$,
  \item[(ii)] $\gamma(\ff) = |\ff(\bar 1)| = |\ff|-d_{\ff}(1)$.
\end{itemize}
\end{prop}
In many cases, the following easy result reduces the maximum diversity problem for shifted families to well-known extremal results.

\begin{lem}[\cite{Fra3}]\label{lem25} Suppose that $\ff\subset 2^{[n]}$ is $r$-wise $t$-intersecting and shifted. Then $\ff(\bar 1)$ is $r$-wise $(t+r-1)$-intersecting.
\end{lem}
\begin{proof}
  Assume for contradiction that there exist $G_1,\ldots, G_r\in \ff(\bar 1)$ satisfying $G_1\cap \ldots \cap G_r = \{x_1,\ldots, x_{r+t-2}\}$. Set $F_i:=(G_i\setminus \{x_i\})\cup \{1\}$ for $i=1,\ldots, r-1$ and $F_r:=G_r$. Then $F_i\in \ff$ by shiftedness and
  $$F_1\cap \ldots \cap F_r = \{x_r,\ldots, x_{r+t-2}\},$$
  contradicting the $r$-wise $t$-intersecting property.
\end{proof}
\begin{cor}\label{cor13}
  Suppose that $n\ge 1+(k-t)(t+2)$ and $\ff\subset {[n]\choose k}$ is $2$-wise $t$-intersecting and shifted. Then
  \begin{equation}\label{eq21}
    \gamma(\ff)\le {n-t-2\choose k-t-1}.
  \end{equation}
\end{cor}
\begin{proof}
  Since $\gamma(\ff)=|\ff(\bar 1)|$ by Proposition~\ref{prop24} (ii), we need to show $|\ff(\bar 1)|\le {n-t-2\choose k-t-1}$. By Lemma~\ref{lem25}, the family $\ff(\bar 1)\subset {[2,n]\choose k}$ is $2$-wise $(t+1)$-intersecting. Now \eqref{eq21} follows from the Complete Erd\H os-Ko-Rado Theorem \cite{Fra78, Wil}.
\end{proof}
\begin{cor}\label{cor14}
  Suppose that $r\ge 3$, $t\le 2^r-2r$ and $\ff\subset 2^{[n]}$ is $r$-wise $t$-intersecting and shifted. Then
  \begin{equation}\label{eq22} \gamma(\ff)\le 2^{n-r-t}.
  \end{equation}
\end{cor}
  The family $$\aaa:=\big\{A\subset [n]\ :\ |A\cap [r+t]|\ge r+t-1\big\}$$ shows that \eqref{eq22} is best possible for $n\ge r+t$.
\begin{proof}
Note that $\ff(\bar 1)$ is $r$-wise $(r+t-1)$-intersecting and $r+t-1\le 2^r-r-1$ by our assumption. Thus \eqref{eq22} follows from the bound on the size of $r$-wise $(r+t-1)$-intersecting families, proved in \cite{F91}.
\end{proof}
Let us also remark that for $n<r+t$ one has $\gamma(\ff)=0$ whenever $\ff$ is $r$-wise $t$-intersecting. For the results on diversity of $r$-wise $t$-intersecting families $\ff\subset 2^{[n]}$ that are not necessarily shifted, see Section~\ref{sec7}.

In the next lemma we show that any two shifted families are positively correlated. (And thus behave similarly to {\it up-sets} in this respect.) It proved to be very useful in our proof. We believe that it will have more applications to other problems involving shifted families.
\begin{lem}[``Shifted families are positively correlated'']\label{lemcorrshift} Let $\ff_1,\ff_2\subset {[n]\choose k}$ be shifted. Then \begin{equation}\label{eqcorrshift}
|\ff_1\cap \ff_2|\ge |\ff_1||\ff_2|/{n\choose k}.
\end{equation}
\end{lem}
\begin{proof}
  Let us apply induction on $k$, and for fixed $k$ induction on $n$. The case $n=k$ is trivial, as well as the case $k=1$. Suppose that $n>k>1$ and that \eqref{eqcorrshift} holds for $n-1,k-1$ and $n-1,k$.

  Note that $\ff_i(1)\subset {[2,n]\choose k-1}$ and $\ff_i(\bar 1)\subset {[2,n]\choose k}$, $i=1,2$. These families are shifted. Therefore, by induction we have
  \begin{equation}\label{eqshift1}
    |\ff_1(1)\cap \ff_2(1)| \ge  \frac{|\ff_1(1)||\ff_2(1)|}{{n-1\choose k-1}},\ \ \ \ \ \ \ \
    |\ff_1(\bar 1)\cap \ff_2(\bar 1)| \ge  \frac{|\ff_1(\bar 1)||\ff_2(\bar 1)|}{{n-1\choose k}}.
  \end{equation}
Due to shiftedness, $\Delta(\ff_i) = |\ff_i(1)|$ for $i=1,2,$ and so \begin{equation}\label{eqshift3}\frac{|\ff_i(1)|}{{n-1\choose k-1}}\ge \frac {|\ff_i(\bar 1)|}{{n-1\choose k}}.
   \end{equation}
We have
\begin{align*}|\ff_1\cap \ff_2| =& |\ff_1(1)\cap \ff_2(1)| +|\ff_1(\bar 1)\cap \ff_2(\bar 1)| \overset{\eqref{eqshift1}}{\ge} \frac{|\ff_1(1)||\ff_2(1)|}{{n-1\choose k-1}}+\frac{|\ff_1(\bar 1)||\ff_2(\bar 1)|}{{n-1\choose k}}\\
=&\frac{|\ff_1||\ff_2|}{{n\choose k}}+\frac{\big({n-1\choose k}|\ff_1(1)|-{n-1\choose k-1}|\ff_1(\bar 1)|\big)\cdot\big({n-1\choose k}|\ff_2(1)|-{n-1\choose k-1}|\ff_2(\bar 1)|\big)}{{n\choose k}{n-1\choose k}{n-1\choose k-1}}\\
\overset{\eqref{eqshift3}}{\ge}& \frac{|\ff_1||\ff_2|}{{n\choose k}},\end{align*}
where the equality between the first and second lines is easy to verify by direct computation (multiplying out the brackets in the latter expression).
\end{proof}

\section{Advanced tools}\label{sec3}

It was first observed by Daykin \cite{Day3} that Theorem~\ref{thmkk} implies Theorem~\ref{thmekrcross}.
Let us show this relationship and restate Theorem~\ref{thmkk} in terms of cross-intersecting families.
\begin{cor}\label{corkk}
  Let $n\ge a+b$. If $\aaa\subset {[n]\choose a}$ and $\bb\subset {[n]\choose b}$ are cross-intersecting and $|\aaa|\ge {x\choose n-a}$ for some $n-a\le x\le n$, then $|\bb|\le {n\choose b}-{x\choose b}$.
\end{cor}
Although it is rather standard, we give the proof for completeness.
\begin{proof}
  Consider the family $\aaa^c:=\{[n]\setminus A: A\in\aaa\}$. We have $|\aaa^c|=|\aaa|\ge {x\choose n-a}$, and, using Theorem~\ref{thmkk}, we get $|\partial^{b}\aaa^c|\ge {x\choose b}$ (note that, in case $n=a+b$, $\partial^b\aaa^c = \aaa^c$ and thus we do not even need Theorem~\ref{thmkk}). It should be clear that $\bb$ and $\partial^{b}\aaa^c$ are disjoint, which implies the bound in the corollary.
\end{proof}

Recall that ``$\oplus$'' stands for symmetric difference. Let us next state the Kruskal--Katona theorem in terms of (co)lexicographic order. For two sets $A,B\subset {[n]\choose k}$ we say that $A\prec_{lex} B$ or that $A$ {\it precedes $B$ in the lexicographic order}, if the smallest element in $A\oplus B$ belongs to $A$. Similarly, $A\prec_{colex} B$ if the largest element in $A\oplus B$ belongs to $B$. We note that if $A\prec_s B$, then both $A\prec_{lex} B$ and $A\prec_{colex} B$. If for some $t$ $\ff\subset {[n]\choose k}$ consists of the $t$ first sets in the lexicographic order, then we call $\ff$ an {\it initial segment (in the lexicographic order)}, and similarly for the colexicographic order. For a set $X$ and integers $t,k$, let us denote $\mathcal L(X,t,k)$ the family of the first $t$ $k$-element subsets of $X$ in the lexicographic order ({\it the initial segment of size $t$ in ${X\choose k}$}), and define  $\mathcal {C}(X,t,k)$ similarly for colexicographic order.

The following is the classic form of the Kruskal-Katona theorem.
\begin{thm}[\cite{Kr,Ka}]\label{thmkk2}
  If $\ff\subset {[n]\choose k}$, then $|\partial \ff|\ge \partial \mathcal {C}([n],|\ff|,k)$.
\end{thm}
Similarly, one can derive a slightly weaker version of Theorem~\ref{thmkk2} in terms of cross-intersecting families. This was observed by Hilton \cite{Hil}.
\begin{cor}\label{corkk2}
  Let $n\ge a+b$ and assume that $\aaa\subset {[n]\choose a}$, $\bb\subset {[n]\choose b}$ are cross-intersecting. Then $\mathcal L([n],|\aaa|,a)$ and $\mathcal L([n],|\bb|,b)$ are cross-intersecting.
\end{cor}

Fix any two disjoint sets $U,V$ of the same size. The following operation, which generalizes shifting, was introduced by Daykin \cite{Day}.  For a set $F$, the {\it $U\leftarrow V$ shift} $S_{U,V}(F)$ is $S_{U,V}(F):=(F\setminus V)\cup U$ if $F\cap (V\cup U)=V$, and $S_{U,V}(F):=F$ otherwise. For a family $\ff\subset 2^{[n]}$,  define
$$S_{U,V}(\ff):=\big\{S_{U,V}(F):F\in\ff\big\}\cup \big\{F:F,S_{U,V}(F)\in\ff\big\}.$$

Daykin used it in the following setting.
Take integers $n>k>0$ and consider a family $\ff\subset {[n]\choose k}$. If $\ff$ is not an initial segment in the colex order, then there exists a pair $(U,V)$ of disjoint sets, such that $|U|=|V|$, $U\prec_{colex} V$ and a set $F\in \ff$ satisfying $F\cap (U\cup V) = V$, such that $(F\setminus V)\cup U$ is not in $\ff$. Take such pair $(U,V)$ that is inclusion-minimal and apply $S_{U,V}$ to $\ff$. To underline the choice of $U$ and $V$, we denote any such $S_{U,V}$-shift by $S_{U,V}^{colex}$.

The next lemma is the essence of Daykin's short and elegant proof of the Kruskal--Katona Theorem.

\begin{lem}\label{lemday2}
  We have $\big|\partial S_{U,V}^{colex}(\ff)\big|\le |\partial \ff|$.
\end{lem}

We can define a similar operation in the cross-intersecting setting.
Take integers $a,b,n$, where $n\ge a+b$, and cross-intersecting families $\aaa\subset {[n]\choose a}$, $\bb\subset {[n]\choose b}$. Assume that $\aaa$  and $\bb$ are not initial segments in the lexicographic order. Then we can find members $A\in\aaa$, $B\in \bb$ and pairwise disjoint sets $(U,V)$ and $(U',V')$ such that $|U|=|V|,$ $|U'|=|V'|$, $U\prec_{lex} V$, $U'\prec_{lex} V'$, $A\cap (U\cup V)=V$, $B\cap (U'\cup V')=V'$ and the sets $(A\setminus V)\cup U$, $(B\setminus V')\cup U'$ are not in $\aaa$ and $\bb$, respectively.
Let us choose the sets that are inclusion-minimal and assume by symmetry that $|U|\le |U'|$.
Then the operation $S_{U,V}$ for any such pair $U,V$ is denoted by  $S_{U,V}^{lex}$.

The next lemma is essentially Lemma~\ref{lemday2} restated in cross-intersecting terms.
\begin{lem}\label{lemday}
  If $\aaa$ and $\bb$ are cross-intersecting and $U,V$ are as above, then the families $S_{U,V}^{lex}(\aaa)$ and $S_{U,V}^{lex}(\bb)$ are cross-intersecting as well.
\end{lem}
Since Daykin did not state this statement in the above form, we provide the proof. Note the role of $U,V$ being inclusion-minimal. We also note that in the symmetric cross-intersecting setting, the proof becomes slightly easier.
\begin{proof}
  Arguing for contradiction, we may assume by symmetry that there is $A\in S_{U,V}^{lex}(\aaa)\setminus \aaa$ and $B\in S_{U,V}^{lex}(\bb)$ that are disjoint. Now $A\cap (U\cup V)=U$ implies $B\cap U=\emptyset$, in particular $B\in \bb$. Note also that $A\in S_{U,V}^{lex}(\aaa)\setminus \aaa$ implies that $\tilde A:=(A\setminus U)\cup V$ is in $\aaa\setminus S_{U,V}^{lex}(\aaa)$.

  We distinguish two cases.

  (i) $V\subset B$. By definition, $S_{U,V}(B) = (B\setminus V)\cup U\in \bb$ follows. However, $S_{U,V}^{lex}(B)\cap \tilde A=\emptyset$, contradicting the cross-intersection property of $\aaa$ and $\bb$.

  (ii) $V\not \subset B$. Define $\tilde V:=B\cap V$, a proper subset of $V$. Let $\tilde U\subset U$ be the subset of $U$ that consists of $|\tilde V|$ smallest elements of $U$. Then $\tilde U\prec_{lex} \tilde V$.\footnote{Note that this would not necessarily hold for colex.}

  By the  choice of $U$ being inclusion-minimal, $\tilde B:=(B\setminus \tilde V)\cup \tilde U \in \bb$. However, $\tilde B\cap V=\emptyset$, $\tilde A\cap U=\emptyset$ and on $[n]\setminus (U\cup V)$ $\tilde A$ and $\tilde B$ coincide with $A$ and $B$. Consequently, $\tilde A\cap \tilde B=\emptyset$, the final contradiction.
\end{proof}

In the proof of Theorem~\ref{thmmain1}, we shall need the following lemma, resemblant of \cite[Lemma~2]{FT}.

\begin{lem}\label{lemkk} Fix integers $s\ge 2, t\ge 2$ and $m\ge s+t-1$.
\begin{itemize}
\item[(i)] The function $f(x,m,t,s):={x\choose t-1}\cdot\frac{{m-3\choose s-2}}{{m-3\choose t-2}}-{x\choose m-s}$ is monotone increasing for $m-s\le x\le m-3$, and is strictly monotone increasing if additionally $m\ge s+t$.
\item[(ii)]Moreover, if $s\ge 3$ and $f(x,m,t,s)\ge f(m-3,m,t,s)$ for some $x>m-3$, then $f(y,m,t,s)\ge f(m-3,m,t,s)$ for any $m-3\le y\le x$.
\item[(iii)] In particular,  $f(y,m,t,s)\ge f(m-3,m,t,s)$ for any $m-3\le y\le m-2$ and $s\ge 2$.
\end{itemize}
\end{lem}
\begin{proof} For $m=s+t-1$ $f$ is constant as a function of $x$. Thus, (i)--(iii) are obvious in this case. In what follows, we assume that $m\ge s+t$.

(i) 
For $s=2$ there is nothing to prove. Assume that $x\ge m-s$ and $s\ge 3$. We have
$$f'(x,m,t,s)={x\choose t-1}\cdot\frac{{m-3\choose s-2}}{{m-3\choose t-2}}\cdot\sum_{i=0}^{t-2}\frac 1{x-i}-{x\choose m-s}\cdot \sum_{i=0}^{m-s-1}\frac 1 {x-i}.$$
Thus, proving $f'(x,m,t,s)> 0$ is equivalent to proving
\begin{equation}\label{eqkk1}
  \frac{{x\choose t-1}{m-3\choose s-2}}{{x\choose m-s}{m-3\choose t-2}}> 1+\frac{\sum_{i=t-1}^{m-s-1}\frac 1 {x-i}}{\sum_{i=0}^{t-2}\frac 1{x-i}}.
\end{equation}
The right hand side of \eqref{eqkk1} is strictly smaller than
$$g(x,m,t,s):=1+\frac{(m-s-t+1)/(x-(m-s-1))}{(t-1)/x}=1+\frac{(m-s-t+1)x}{(x-m+s+1)(t-1)}.$$
It is clear that $g(x,m,t,s)\cdot (x-m+s+1)$ is monotone increasing as $x$ increases. On the other hand,
$$\frac{{x\choose t-1}{m-3\choose s-2}}{{x\choose m-s}{m-3\choose t-2}}\cdot (x-m+s+1)=\frac{{m-3\choose s-2}}{{m-3\choose t-2}}\cdot\frac{(m-s)(m-s-1)\ldots t}{(x-t+2)(x-t+1)\ldots(x-m+s+2)},$$
and thus the right hand side of \eqref{eqkk1} times $(x-m+s+1)$ is monotone decreasing as $x$ grows. Thus, in order to verify \eqref{eqkk1} for any $m-s\le x\le m-3$, it is sufficient to show that
$$\frac{{x\choose t-1}{m-3\choose s-2}}{{x\choose m-s}{m-3\choose t-2}}\ge 1+\frac{(m-s-t+1)x}{(x-m+s+1)(t-1)}$$
for $x=m-3$. Substituting the value of $x$ and rewriting, we obtain the inequality
\begin{align*}
   & \frac{{m-3\choose t-1}{m-3\choose s-2}}{{m-3\choose s-3}{m-3\choose t-2}}-\frac{(m-s-t+1)(m-3)}{(s-2)(t-1)}\ge 1 \\
  \Leftrightarrow & \frac{(m-t-1)(m-s)}{(s-2)(t-1)}-\frac{((m-s)-(t-1))((m-t-1)+(t-2))}{(s-2)(t-1)}\ge1 \\
  \Leftrightarrow &   \frac{(t-1)((m-t-1)+(t-2))-(m-s)(t-2)}{(s-2)(t-1)}\ge1
  \\
  \Leftrightarrow & \frac{(t-1)(s-3)+(m-s)}{(s-2)(t-1)}\ge1.
\end{align*}
The last inequality clearly holds, since $m-s\ge t-1$.

(ii) Assume that this is not the case. Then there are $p,q$ such that $m-3<p<q\le y$, $f(q,m,t,s)<f(m-3,m,t,s)<f(p,m,t,s)$ and $f'(q,m,t,s)=f'(p,m,t,s)=0$. Whenever $f'(z,m,t,s)=0$, we have
$${z\choose t-1}\cdot\frac{{m-3\choose s-2}}{{m-3\choose t-2}}={z\choose m-s}\cdot \frac{\sum_{i=0}^{m-s-1}\frac 1 {z-i}}{\sum_{i=0}^{t-2}\frac 1{z-i}},$$
and thus $f(z,m,t,s)=h(z,m,t,s)$, where
$$h(z,m,t,s) = {z\choose m-s}\cdot \frac{\sum_{i=t-1}^{m-s-1}\frac 1 {z-i}}{\sum_{i=0}^{t-2}\frac 1{z-i}}.$$
It is easy to see that $h(z,m,t,s)$ is a monotone increasing function of $z$,\footnote{in fact, even $h(z,m,t,s)\cdot \sum_{i=0}^{t-2}\frac 1{z-i}$ is monotone increasing} and thus  $f(m-3,m,t,s)< f(p,m,t,s)=h(p,m,t,s)\le h(q,m,t,s)=f(q,m,t,s)$, contradicting our assumption.

(iii) In view of (ii), for $s\ge 3$ it is sufficient to verify that $f(m-3,m,t,s)\le f(m-2,m,t,s)$. We have
$$f(m-2,m,t,s)-f(m-3,m,t,s) = {m-3\choose t-2}\cdot \frac{{m-3\choose s-2}}{{m-3\choose t-2}}-{m-3\choose m-s-1} =0.$$
For $s=2$ and $y<m-2$ this is obvious. The case $s=2,y=m-2$ is a straightforward computation.
\end{proof}

The following theorem was proven by Frankl, Lee, Siggers and Tokushige \cite{FLST}.
\begin{thm}[\cite{FLST}]\label{thmflst} For every $k\ge t\ge 14$ and $n\ge (t+1)k$ we have the following. If  $\aaa,\bb\subset {[n]\choose k}$ are cross $t$-intersecting, then $|\aaa||\bb|\le {n-t\choose k-t}^2$.
\end{thm}
We can derive the following corollary for $t<14$.
\begin{cor}\label{corflst} Fix some  $k\ge t>0$ and an integer $\alpha$. If the statement of Theorem~\ref{thmflst} is valid for $k+\alpha,t+\alpha$ and $n+\alpha\ge n_0+\alpha$ playing the roles of $k,t,n$, then it is valid for $k,t$ and $n\ge n_0$. In particular, Theorem~\ref{thmflst} is valid for any $k\ge t\ge 1$ and $n\ge \max\{15k, (t+1)k\}$.
\end{cor}
\begin{proof} Given cross $t$-intersecting families $\aaa,\bb\subset{[n]\choose k}$, we obtain families  $\aaa',\bb'\subset{[n+\alpha]\choose k+\alpha}$, where $\aaa':=\{A\cup [n+1,n+\alpha]: A\in \aaa\}$ and $\bb'$ is defined similarly. Then, clearly, $\aaa',\bb'$ are cross $(t+\alpha)$-intersecting, and thus $|\aaa'||\bb'|\le {(n+\alpha)-(t+\alpha)\choose (k+\alpha)-(t+\alpha)}^2 = {n-t\choose k-t}^2$ since the conditions of Theorem~\ref{thmflst} hold for $\aaa', \bb'$. Since $|\aaa|=|\aaa'|$ and $|\bb|=|\bb'|$, the conclusion of the theorem is valid for $\aaa,\bb$ as well. The second part of the statement follows directly from the first one and Theorem~\ref{thmflst}.
\end{proof}

\section{Cross-intersecting families. Proof of Theorem~\ref{thmmain1}}\label{sec5}

   Take $u,v$ and $\aaa,\ \bb$ as in the statement.
The proof of the theorem consists of the following two lemmas. The first one handles the case when the diversities of the families $\aaa,\bb$ are not too large.
  \begin{lem}\label{lemproof1}
        Theorem~\ref{thmmain1} is valid, provided that additionally $\gamma(\aaa)={n-u'-1\choose n-a-1}$ and $\gamma(\bb) = {n-v'-1\choose n-b-1}$ for some $u', v'\ge 2$.
  \end{lem}
   \begin{proof} Recall that we put ${x\choose k}:=0$ for $x<k$.
Arguing indirectly, assume that $u'\le u.$   W.l.o.g., suppose that %$|\aaa| = {n-1\choose a-1}-{n-v-1\choose a-1}+{n-u-1\choose n-a-1}$ and that
element $1$ has the largest degree in $\aaa$.  Clearly, $\aaa(1)$ and $\bb(\bar 1)$, as well as $\bb(1)$ and $\aaa(\bar 1)$, are cross-intersecting.

By Corollary~\ref{corkk} we conclude
that $|\bb(1)|\le {n-1\choose b-1}-{n-u'-1\choose b-1}$. Since $|\bb|\ge {n-1\choose b-1}-{n-u-1\choose b-1}+{n-v-1\choose n-b-1}$, we have $v'\le v$. Thus, diversity of $\bb$ is non-zero and we similarly have $|\aaa(1)|\le {n-1\choose a-1}-{n-v'-1\choose a-1}$.

Let us consider the following expression:
$$E:=\Big({n-1\choose a-1}-|\aaa|\Big)+\frac{{n-4\choose a-2}}{{n-4\choose b-2}}\Big({n-1\choose b-1}-|\bb|\Big).$$
On the one hand, by our assumptions on $|\aaa|$ and $|\bb|$,
\begin{equation}\label{eqf1} E< \Big({n-v-1\choose a-1}-{n-u-1\choose n-a-1}\Big)+\frac{{n-4\choose a-2}}{{n-4\choose b-2}}\Big({n-u-1\choose b-1}-{n-v-1\choose n-b-1}\Big).\end{equation}
On the other hand, due to the bounds on $|\aaa(1)|$ and $|\bb(1)|$ above, we have
\begin{small}\begin{align*}E\ \ge &\ \Big({n-v'-1\choose a-1}-\frac{{n-4\choose a-2}}{{n-4\choose b-2}}{n-v'-1\choose n-b-1}\Big)+\Big(\frac{{n-4\choose a-2}}{{n-4\choose b-2}}{n-u'-1\choose b-1} - {n-u'-1\choose n-a-1}\Big)\\
=&\ \frac{{n-4\choose a-2}}{{n-4\choose b-2}}f(n-v'-1,n-1,a,b)+f(n-u'-1,n-1,b,a),
\end{align*}\end{small}
where the function $f$ is as in Lemma~\ref{lemkk}. We have shown in Lemma~\ref{lemkk} that $f(x,m,t,s)$  increases as $x$ increases, as long as $m-s\le x\le m-3$, and that $f(x,m,t,s)\ge f(m-3,m,t,s)$ for $m-3\le x\le m-2$. This, combined with the assumption $2\le v'\le v$, and $2\le u'\le u$, allows us to conclude
\begin{align}E\ \ge\ & \frac{{n-4\choose a-2}}{{n-4\choose b-2}}f(n-v-1,n-1,a,b)+f(n-u-1,n-1,b,a) \notag\\
\label{eqf2} =\ & \Big({n-v-1\choose a-1}-{n-u-1\choose n-a-1}\Big)+\frac{{n-4\choose a-2}}{{n-4\choose b-2}}\Big({n-u-1\choose b-1}-{n-v-1\choose n-b-1}\Big).\end{align}
Inequalities \eqref{eqf1} and \eqref{eqf2} contradict each other, which implies that our assumption on $u'$ was wrong, and thus $\gamma(\aaa)<{n-u-1\choose n-a-1}$. The same holds for $\bb$. Finally, note that we have actually shown that $|\aaa(\bar 1)|< {n-u-1\choose a-1}$, and thus $|\bb(\bar 1)|< {n-v-1\choose b-1}$, which implies that $1$ has the largest degree in both families.
\end{proof}

The second lemma allows us to reduce the case of (at least) one of the families having large diversity to the case handled in Lemma~\ref{lemproof1}.

\begin{lem}\label{lemproof2} Given $\aaa$, $\bb$ satisfying the conditions of Theorem~\ref{thmmain1}, there exist cross-intersecting families $\aaa',$ $\bb'$ with $\gamma(\aaa')\le {n-3\choose a-2}$, $\gamma(\bb)\le {n-3\choose b-2}$ and such that $|\aaa'|\ge |\aaa|$, $|\bb'|\ge |\bb|$. Moreover, either $\gamma(\aaa')\ge\min\{{n-u-1\choose n-a-1},\gamma(\aaa)\}$ or $\gamma(\bb')\ge \min\{{n-v-1\choose n-b-1},\gamma(\bb)\}$.%\footnote{Note that, for $x\in \R$, $\lfloor x\rfloor +1$ is the least integer strictly greater than $x$.}
\end{lem}
It should be clear that Lemmas~\ref{lemproof2} and~\ref{lemproof1} imply Theorem~\ref{thmmain1}. We prove Lemma~\ref{lemproof2} using a series of transformations that would maintain the necessary properties of $\aaa$, $\bb$. The transformations would include shifts and Daykin's shifts, as well as some others. The main difficulty here would be to show that the diversities of $\aaa, \bb$ are not decreasing by too much after each transformation (i.e., to maintain the moreover condition from Lemma~\ref{lemproof2}). Another aspect that we will have to keep track of is that the transformation process will eventually converge to the desired situation. This is done by, roughly speaking, showing that each transformation makes our families closer to initial segments in lex order. Since the proof of Lemma~\ref{lemproof2} is rather involved and consists of several reduction steps, we present it in a separate subsection.

\subsection{Proof of Lemma~\ref{lemproof2}}
Arguing indirectly, among all the pairs $\aaa,\bb$ that do not satisfy the statement of Lemma~\ref{lemproof2}, let us take the pair that is minimal in the following sense: there is no pair $\aaa_1,\bb_1$ other than $\aaa,\bb$ violating the statement of Lemma~\ref{lemproof2} and such that $\aaa_1 \prec' \aaa$, $\bb_1\prec'\bb$, where $\ff'\prec'\ff$ if the sum of order numbers of sets in $\ff'$ in lex order is smaller than that of $\ff$.

{\sc Case~1. } Assume that  $\Delta (\aaa)\le {n-2\choose a-2}+{n-u-1\choose n-a-1},$ $\Delta(\bb)\le {n-2\choose b-2}+{n-v-1\choose n-b-1}$. In particular, this implies that
$$\gamma(\aaa)\ge {n-3\choose a-2}+{n-4\choose a-2}, \ \ \  \gamma(\bb)\ge{n-3\choose b-2}+{n-4\choose b-2}.$$
Indeed, \begin{footnotesize}$$\gamma(\aaa)=|\aaa|-\Delta(\aaa)\ge{n-1\choose a-1}-{n-4\choose a-1}+{n-u-1\choose n-a-1}-{n-2\choose a-2}-{n-u-1\choose n-a-1}= {n-3\choose a-2}+{n-4\choose a-2}.$$
  \end{footnotesize}

 In what follows, we obtain a contradiction with the minimality of $\aaa,\bb$, given that the inequalities defining Case~1 hold.
 Then we proceed to the next case.   We first ``test'' the families $\aaa,\bb$ using shifting.

W.l.o.g., assume that $a\le b$ (and thus $n\ge 2a$).

For any $i<j$, if $\gamma(S_{ij}(\aaa))\ge {n-4\choose a-3}$ or $\gamma(S_{ij}(\bb))\ge {n-4\choose b-3}$ then we call a shift {\it successful}. After a successful shift, we obtain cross-intersecting families $S_{ij}(\aaa),S_{ij}(\bb)$. Moreover, if Lemma~\ref{lemproof2} fails for $\aaa,\bb$, then it fails for $S_{ij}(\aaa),S_{ij}(\bb)$. Indeed, if not, then the families $\aaa',\bb'$ that we found will also fit for $\aaa$ and $\bb$. (The only nontrivial condition to check is the ``moreover'' condition. If, say, $\gamma(S_{ij}(\aaa))\ge {n-4\choose k-3}$, then $\gamma(\aaa')\ge \min\{{n-u-1\choose n-a-1},{n-4\choose k-3}\}\ge {n-u-1\choose n-a-1}$, where the last inequality is due to $u>3$.)

If all the shifts are successful (and we always stay in Case~1), then, by minimality, $\aaa,\bb$ are shifted.

Now, assume that an $i\leftarrow j$ shift is unsuccessful.  If this is the case then $\gamma(\aaa)-\gamma(S_{ij}(\aaa))> 2{n-4\choose a-2}$ and the same for $\bb$. Then, using Proposition~\ref{prop55} (iii), we get that the following properties hold for both $(\ff,k)\in \{(\aaa,a),(\bb,b)\}$.
\begin{itemize}
  \item[(i)] $|\ff(\bar i\bar j)|<{n-4\choose k-3}$;
  \item[(ii)] $|\ff(i\bar j)|, |\ff(j\bar i)|> 2{n-4\choose k-2}$.
  \item[(iii)] $|\ff(i\bar j)\oplus \ff(\bar i j)|> 4{n-4\choose k-2}$;
  \item[(iv)] $|\ff(i\bar j)\cap \ff(\bar i j)|<{n-4\choose k-3}$.
\end{itemize}
Here (iv) follows from the fact that the sets in $\ff(i\bar j)\oplus \ff(\bar i j)$ stay intact after the shift and thus provide a lower bound on the diversity for $S_{ij}(\ff)$. %If a shift is unsuccessful, then we do not perform it, and try to perform other shifts instead. Note also that after some steps an unsuccessful shift may become successful, and then we may perform it.
The following lemma is crucial for the understanding of the structure of the family of all unsuccessful  shifts.

\begin{lem}\label{intershift} If $i\leftarrow j$ shift is unsuccessful, then any $k\leftarrow l$ shift must be successful for $\{i,j\}\cap \{k,l\}=\emptyset$.
\end{lem}
\begin{proof} Using (iii), we have $$|\aaa(i\bar j)\oplus \aaa(\bar i j)|> 4{n-4\choose k-2}.$$
Recall that $n\ge 2a$. Then \begin{small}\begin{equation*}\label{eqdiff} \big|\aaa(i\overline{jk})\oplus \aaa(j\overline{ik})\big|\ge\Big|\big(\aaa(i\bar j)\oplus \aaa(j\bar i)\big)\setminus {[n]\setminus\{i,j,k\}\choose a-2}\Big|> 4{n-4\choose a-2}-{n-3\choose a-2}\ge 2{n-4\choose a-2},\end{equation*}\end{small}
since ${n-4\choose a-2}\ge {n-4\choose a-3}$ for $n\ge 2a$. Similarly,
$$\big|\aaa(i\overline{jkl})\oplus \aaa(j\overline{ikl})\big|\ge\big|\big(\aaa(i\overline{jk})\oplus \aaa(j\overline {ik})\big)\big|-{n-4\choose k-2}> {n-4\choose a-3},$$
and so (i) is not satisfied for the $k\leftarrow l$ shift. Thus, it must be successful.
\end{proof}

Lemma~\ref{intershift} implies that either the families $\aaa,\bb$ are shifted or there is an unsuccessful shift, say $1\leftarrow 2$, and in that case both $\aaa(1\bar 2)$ and $\aaa(\bar 1 2)$ are shifted. Let us show that the second situation is not possible. Indeed, apply \eqref{eqcorrshift} to $\aaa(1 \bar 2)$, $\aaa(\bar 1 2)$. We get that
$$\big|\aaa(1\bar 2)\cap \aaa(\bar 1 2)\big|\cdot  {n-2\choose a-1}\overset{\eqref{eqcorrshift}}{\ge} |\aaa(1\bar 2)|\cdot |\aaa(2\bar 1)|\overset{(ii)}{>} 4{n-4\choose a-2}^2\ge {n-3\choose a-2}^2.$$
Note that for the last inequality we used $2a\le n$. But ${n-4\choose a-3}{n-2\choose a-1} = \frac {(a-2)(n-2)}{(n-3)(a-1)}\cdot{n-3\choose a-2}^2\le {n-3\choose a-2}$,\footnote{We have $\frac {a-2}{a-1}\le \frac {n-3}{n-2}$ since $a\le n-1$.} and thus $|\aaa(1\bar 2)\cap \aaa(\bar 1 2)|> {n-4\choose k-3},$ a contradiction with property (iii) of an unsuccessful shift.

Recall the definition of a $S_{U,V}^{lex}$-shift from Section~\ref{sec3}. Find $U,V\subset [n]$ as in the definition %(of minimal size among possible choices on $[2,n]$
and consider $S_{U,V}^{lex}(\aaa)$, $S_{U,V}^{lex}(\bb)$. Put $m:=|U|$. Clearly, $m>0$ since $\aaa,\bb$ do not form initial segments in the lexicographical order due to small maximal degree. Moreover, $m\ne 1$ since the families $\aaa,\bb$ are shifted.  Therefore, $m\ge 2$. First, $|\aaa|=|S_{U,V}^{lex}(\aaa)|,|\bb|=|S_{U,V}^{lex}(\bb)|.$ Next, we  move at most ${n-|U|-|V|\choose a-|U|}\le {n-4\choose k-2}$ sets to obtain $S_{U,V}^{lex}(\aaa)$ from $\aaa$. Therefore, we have $\gamma(S_{U,V}^{lex}(\aaa))\ge {n-4\choose k-3}$. Finally, Lemma~\ref{lemday} implies that $S_{U,V}^{lex}(\aaa)$ and $S_{U,V}^{lex}(\bb)$ are cross-intersecting.
Thus, we get a contradiction with the minimality of $\aaa,\bb$.

{\sc Case~2. } Either $\Delta(\aaa)> {n-2\choose a-2}+{n-u-1\choose n-a-1}$ or $\Delta(\bb)>{n-2\choose b-2}+{n-v-1\choose n-b-1}$. Without loss of generality, let the first inequality hold and assume that $\Delta(\aaa)=|\aaa(1)|$. Define $\aaa_1,\bb_1$ such that $\aaa_1(1)= \mathcal L([2,n],|\aaa(1)|,k-1)$, $\aaa_1(\bar 1)=\mathcal L([2,n],|\aaa(\bar 1)|,k)$, and similarly for $\bb_1$. (That is, $\aaa_1(1),\aaa(\bar 1),\bb(1),\bb(\bar 1)$ all form lex segments on $[2,n]$.)

We may assume that $\aaa=\aaa_1,\bb=\bb_1$. Indeed, $|\aaa|=|\aaa_1|,$ $|\bb|=|\bb_1|$ and, due to Corollary~\ref{corkk2}, the pairs $(\aaa_1(1),\bb_1(\bar 1))$, $(\aaa_1(\bar 1),\bb_1(1))$ and $(\aaa_1(\bar 1),\bb_1(\bar 1))$ are cross-intersecting, implying that $\aaa_1,\bb_1$ are cross-intersecting. Moreover, $\gamma(\aaa_1)\ge \min\{|\aaa(\bar 1)|, |\aaa(1\bar 2)|\}=\min\{\gamma(\aaa),\lfloor{n-u-1\choose n-a-1}\rfloor+1\big\}.$ Therefore, either $\aaa_1,\bb_1$ may be taken as $\aaa',\bb'$ from the statement of Lemma~\ref{lemproof2}, or we get a contradiction with the minimality of $\aaa,\bb$ unless $\aaa_1=\aaa,$ $\bb_1=\bb$.

Finally, assume that $\aaa,\bb$ have the structure as above, but still, say, $\gamma(\aaa)>{n-3\choose a-2}$. This implies that
$\aaa\supset\big\{A\in {[n]\choose k}: |A\cap [3]|\ge 2\big\}$, in particular, $1\leftarrow j$ shift is successful for any $j\ge 2$ and thus, by minimality, $\aaa$ is shifted.\footnote{Note that $S_{ij}(\aaa)=\aaa$ for any $2\le i<j\le n$, because $\aaa(1)$ and $\aaa(\bar 1)$ are lex segments.}
Since, clearly, $\aaa$ does not form an initial segment in the lex order, there is a Daykin shift that is possible to apply to $\aaa,\bb$, getting ``smaller'' (in the lex sense) families. Arguing as before, the diversity of $\aaa$ will be at least ${n-4\choose a-3}$ after the application of this shift, and thus either we get a contradiction with the minimality of $\aaa,\bb$, or obtain families $\aaa'$, $\bb'$ satisfying the requirements of Lemma~\ref{lemproof2} for $\aaa,\bb$.

\section{Stability for Kruskal--Katona: proof of Theorem~\ref{thmstabkk}}\label{sec42}
Let us first show that the statement is true, provided $\ff\subset {Y\choose k}$ for some $Y$, $|Y|=n+1$. W.l.o.g., assume that $Y=[n+1]$. Consider the families \begin{align*}\label{}
                                   \aaa:=&\ \big\{[n+1]\setminus F: F\in \ff\big\},\\
                                   \bb:=&\ \Big\{F\in {[n+1]\choose k-1}: F\notin \partial (\ff)\Big\}.
                                 \end{align*}
It is easy to see that $|\aaa|=|\ff|$ and $|\bb| = {n+1\choose k-1}-|\partial (\ff)|.$ Moreover, the families $\aaa$ and $\bb$ are cross-intersecting (cf. the proof of Corollary~\ref{corkk}). Putting $m:=n+1,\ a:=n+1-k,\ b:=k-1$, we get that $\aaa\subset {[m]\choose a}$ and $|\aaa|\ge {m-1\choose a-1}- {y\choose a-1}+{x\choose m-a-1}$.  Next, $\gamma(\aaa) = \gamma_{KK}(\ff,n)\ge {x\choose k-1}={x\choose m-a-1}$. Therefore, applying Theorem~\ref{thmmain1} (with $m:=n,$ $u:=m-x-1$, $v:=m-y-1$), we conclude that $|\bb|\le {m-1\choose b-1}-{x\choose b-1}+{y\choose m-b-1}$ (indeed, having ``$>$'' here would contradict Theorem~\ref{thmmain1}). Consequently,
\begin{small}
$$|\partial \ff|= {m\choose b}-|\bb|\ge {m-1\choose b}-{y\choose m-b-1}+{x\choose b-1} = {n\choose k-1}-{y\choose n-k+1}+{x\choose k-2}.$$
\end{small}
Therefore, it is sufficient to show that we can actually assume that $\ff\subset {Y\choose k}$ for some $Y$, $|Y|=n+1$.

The proof here resembles the proof of Lemma~\ref{lemproof2}, however, a direct reduction seems to be impossible. Arguing indirectly, assume that the theorem does not hold for some $\ff$ satisfying the conditions.
Among such $\ff$, take the one that is minimal w.r.t. colex order.\footnote{More precisely, the sum of order numbers of sets from $\ff$ in colex is the smallest.}

{\sc Case 1} Assume that $\gamma_{KK}(\ff)>{n-2\choose k-1}+{n-3\choose k-2}$.
Our first step is to show that $\ff$ is shifted. Borrowing terminology from the proof of Theorem~\ref{thmmain1}, we say that the $(U,V)$-shift is {\it unsuccessful} if $\gamma_{KK}(S_{U,V}(\ff))<{x\choose k-1}$.

Assume that the $i\leftarrow j$ shift is  unsuccessful. This implies that there is a subset $X$, $X=n-1$, $i,j\notin X$, such that \begin{equation}\label{eqkkproof1}\Big|S_{ij}(\ff)\cap {X\cup\{i\}\choose k}\Big|> {n\choose k}-{y\choose n-k-1}\ge {n\choose k}-{n-3\choose k-3},\end{equation}
where the first inequality comes from the assumption on the size of $\ff$ and the second is due to $y\le n-3$.
 We may then assume that the family is shifted on $X$. Indeed, %consider the family $\ff':=\ff\cap {X\cup \{i,j\}\choose k}$.
assume that one of the $i'\leftarrow j'$ shifts is unsuccessful for $i',j'\in X$. Then $\big|S_{i'j'}(\ff)(j')\big|<{n-3\choose k-1}$ and thus $\big|S_{ij}(S_{i'j'}(\ff))(j')\big|< {n-3\choose k-1}<{n-1\choose k-1}-{n-3\choose k-3}$. This implies that
\begin{equation}\label{eqkkproof2}\Big| S_{ij}\big(S_{i'j'}(\ff)\big) \cap {X\cup \{i\}\choose k}\Big|<{n\choose k}-{n-3\choose k-3}.\end{equation}

However,
\begin{equation}\label{eqkkproof3}\Big| S_{ij}\big(S_{i'j'}(\ff)\big) \cap {X\cup \{i\}\choose k}\Big|=\Big| S_{ij}\big(\ff\big) \cap {X\cup \{i\}\choose k}\Big|\end{equation}
due to the fact that $i',j'\in X$, and shifts inside $X\cup \{i\}$ do not alter the size of the family inside $X\cup \{i\}$. Then \eqref{eqkkproof2} and \eqref{eqkkproof3} contradict \eqref{eqkkproof1}.

Next, if $i\leftarrow j$ shift was unsuccessful then $\gamma_{KK}(\ff)-\gamma_{KK}(S_{ij}(\ff))>{n-2\choose k-1}+{n-3\choose k-2}-{x\choose k-1}\ge 2{n-3\choose k-2}$. An obvious analogue of Proposition~\ref{prop55} (iii) in our situation states that \begin{equation}\label{eqkkproof4}\min\big\{|\ff_i\setminus \ff_j|,|\ff_j\setminus \ff_i|\big\}>2{n-3\choose k-2}>{n-3\choose k-2},\end{equation} where $\ff_l:=\{F\subset {X\choose k-1}: F\cup \{l\}\in \ff\}$ for $l\in \{i,j\}.$ At the same time, due to \eqref{eqkkproof1}, we have $|\ff_i\cup \ff_j|>{n-1\choose k-1}-{n-3\choose k-3} = {n-2\choose k-1}+{n-3\choose k-2}$. Thus, putting $t:=|\ff_i\cap \ff_j|$  we get $|\ff_i|+|\ff_j|\ge {n-2\choose k-1}+{n-3\choose k-2}+t$. The combination of the last inequality and \eqref{eqkkproof4} gives $|\ff_i||\ff_j|\ge {n-2\choose k-1}({n-3\choose k-2}+t)$. At the same time, as we have seen above, the families $\ff_i,\ff_j$ are shifted. Using \eqref{eqcorrshift}, we get
\begin{equation}\label{eqkkproof5}t{n-1\choose k-1}\ge |\ff_i||\ff_j| \ge {n-2\choose k-1}\Big({n-3\choose k-2}+t\Big).\end{equation}
Clearly, $g(t):=t{n-1\choose k-1}-{n-2\choose k-1}({n-3\choose k-2}+t)$ is a monotone increasing function of $t$. We claim that for $t\le {n-3\choose k-1}$  we have $g(t)<0$. It is sufficient to show for $t={n-3\choose k-1}$. In that case, we have $g(t) = {n-3\choose k-1}{n-1\choose k-1}-{n-2\choose k-1}^2=\big(\frac{(n-1)(n-k-2)}{(n-3)(n-k)}-1\big){n-2\choose k-1}^2<0$. Therefore, if there is $t$ satisfying the displayed inequality, it must satisfy $t>{n-3\choose k-1}$, and thus we have $\gamma_{KK}(S_{ij}(\ff))\ge t>{n-3\choose k-1}$, a contradiction with the fact that $i\leftarrow j$ shift is unsuccessful.

From now we suppose that $\ff$ is shifted. Since $\ff$ clearly does not form an initial segment in the colex order, there is a Daykin $U\leftarrow V$-shift $S_{U,V}^{colex}$ that will alter $\ff$. Moreover, we have $s:=|U|=|V|\ge 2$ due to shiftedness of $\ff$. Therefore, by minimality of $\ff$, the $U\leftarrow V$ shift must be unsuccessful, and thus there exists $X$, $|X| = n-s$ and $X\cap (U\cup V)=\emptyset$, such that
\begin{equation}\label{eqkkproof666}\min\big\{|\ff_U\setminus \ff_V|,|\ff_V\setminus \ff_U|\big\}>2{n-3\choose k-2},\end{equation} where $\ff_W:=\{F\subset {X\choose k-s}: F\cup \{W\}\in \ff\}$ for $W\in \{U,V\}.$ Putting $t:=|\ff_U\cap \ff_V|$ and using shiftedness, we get the following analogue of \eqref{eqkkproof5}.
\begin{equation}\label{eqkkproof6}t{n-2\choose k-2}\ge t{n-s\choose k-s}=|\ff_U\cap \ff_V|{X\choose k-s}\ge |\ff_U||\ff_V| \ge \Big(2{n-3\choose k-2}+t\Big)^2.\end{equation}
On the one hand, if $2k\ge n$, then we have ${n-3\choose k-1}{n-2\choose k-2} = \frac {(n-k-2)(n-2)}{(k-1)(n-k)}{n-3\choose k-2}^2<2{n-3\choose k-2}^2$, and thus \eqref{eqkkproof6} can be only satisfied for $t>{n-3\choose k-1}$. On the other hand, if $2k<n$ then ${n-2\choose k-2}={n-3\choose k-2}+{n-3\choose k-3}<2{n-3\choose k-2}$, and thus \eqref{eqkkproof6} (or even \eqref{eqkkproof666}) is never satisfied. Therefore, either \eqref{eqkkproof6} cannot be true, or $\gamma_{KK}(S_{U,V}(\ff))>t\ge {n-3\choose k-1}$, which contradicts the assumption that $(U,V)$-shift was unsuccessful.
We conclude that the assumption of Case~1 contradicts minimality of $\ff$.

{\sc Case 2.} Assume that  ${x\choose k-1}\le \gamma_{KK}(\ff)\le {n-2\choose k-1}+{n-3\choose k-2}$. Therefore, there exists $X$, $|X|=n$, such that $\big|\ff\cap {X\choose k}\big|\ge {n\choose k}-{y\choose k-1}+{x\choose k-1}-{n-2\choose k-1}-{n-3\choose k-2} = {n-1\choose k}+{n-3\choose k-3}-{y\choose k-1}+{x\choose k-1}\ge {n-1\choose k}+{x\choose k-1}$. W.l.o.g., assume that $X = [n]$.
In this case, we may assume that $\ff$ is $(U,V)$-shifted for all $U\prec_{colex}V\subset [n]$: indeed, for any of these shifts we have $$\gamma_{KK}(S_{U,V}(\ff))\ge \min\big\{\gamma_{KK}(\ff), \min_{i\in [n]}d_i(S_{U,V}(\ff))\big\}\ge  \min\Big\{\gamma_{KK}(\ff),{x\choose k-1}\Big\}$$ since $\min_{i\in [n]}d_i(S_{U,V}(\ff))\ge |\ff\cap {[n]\choose k}|-{n-1\choose k}\ge {x\choose k-1}$.
In particular, ${[n-1]\choose k}\subset \ff$. Replace $\ff_{\gamma}:=\ff\setminus {[n]\choose k}$ with $\{F\cup \{n+1\}: F\in \mathcal C([n],|\ff_{\gamma}|,k-1)\}$, obtaining a new family $\ff'$. Again, one easily sees that $\gamma_{KK}(\ff')\ge {x\choose k-1}$. Next, $|\partial \ff'|\le |\partial \ff|$. Indeed, the new shadows $\ff'$ add to the shadows of $\ff\cap {[n]\choose k}$ have form $\big\{F: n+1\in F, F\in \partial \mathcal C([n], |\ff_{\gamma}|,k-1)\big\}$. However, $|\partial \mathcal C([n], |\ff_{\gamma}|,k-1)|$ is smaller than $\sum_{i\ge n+1}|\partial \{F\setminus \{i\}: F\in \ff, i\in F\}|$ due to subadditivity of the shadow of $\mathcal C([n],t,k-1)$ as a function of $t$. It follows from the Kruskal--Katona theorem (Theorem~\ref{thmkk2}) itself: we may realize families $\{F\setminus \{i\}: F\in \ff, i\in F\}$ on disjoint ground sets, and then their union has shadow larger than that of $\mathcal C([n],|\ff'|,k-1)$.

 Thus, due to minimality, $\ff = \ff'$, so $\ff\subset {[n+1]\choose k}$ and we arrive at a contradiction.

\textbf{Remark. } What we used here in essentially \cite[Lemma~3.2]{KLo}, and the elegant way of seeing subadditivity is borrowed from \cite[Lemma~3.1(iii)]{KLo}.

\section{Uniform $r$-wise $t$-intersecting families. Proof of Theorem~\ref{thmmain2}}\label{sec4}
Recall that the families $\ff_1,\ldots, \ff_r$ are {\it cross $t$-intersecting} if for any $F_1\in \ff_1,\ldots, F_r\in \ff_r$ we have $|F_1\cap \ldots \cap F_r|\ge t$.

  %For $r=2$ this was proved by the second author \cite{Kup21}. From now on we suppose that $r\ge 3$.

  Take any family $\ff$ satisfying the conditions of the theorem.

  Suppose that $\gamma(\ff)>0$. This implies that $\ff$ is $2$-wise $(r+t-2)$-intersecting. Indeed, if $F_1\cap F_2=\{i_1,\ldots, i_l\}$ for $l\le r+t-3$, then, given that $\gamma(\ff)>0$, for each $j\in[\min\{r-2,l\}]$ we can find $F_{i_j}\in \ff$ such that $i_j\notin F_{i_j}$, and thus $|F_1\cap F_2\cap F_{i_1}\cap\ldots\cap F_{i_l}|\le t-1$.

  Assume that $\ff$ is $2$-wise $(r+t-1)$-intersecting. Then  $\ff(\bar x)$ is $2$-wise $(r+t-1)$-intersecting for any $x\in[n]$, and, by a theorem of Frankl \cite{Fra78}, $\gamma(\ff)\le |\ff(\bar x)|\le {n-r-t\choose k-r-t+1}$ for $n>2(r+t)k$.

  Assume now that there exist $F_1, F_2\in \ff$ that intersect in $r+t-2$ elements. W.l.o.g., suppose that $F_1\cap F_2=[r+t-2]$. The $r$-wise $t$-intersecting property implies \begin{equation}\label{eqrwise1}|F\cap [r+t-2]|\ge r+t-3\ \ \ \ \ \text{for all }F\in \ff. \end{equation}
  For each $i\in [r+t-2]$, define $$\ff_i:=\big\{F\in \ff: F\cap [r+t-2]=[r+t-2]\setminus \{i\}\big\}.$$
  The families $\ff_1\ldots, \ff_{r+t-2}$ are cross 2-intersecting subfamilies of ${[r+t-1,n]\choose k-r-t+3}$. By \eqref{eqrwise1},
  \begin{equation}\label{eqrwise2} \gamma(\ff)\le \min_i|\ff_i|.
  \end{equation}
  Applying Corollary~\ref{corflst} to $\ff_1,\ff_2$, we get that one of them has size at most ${n-r-t\choose k-r-t+1}$, since $n\ge 15 k> r+t-2 +15(k-r-t+3)$.

\section{Families with $\nu(\ff)\le s$}\label{sec6}

For $X\subset[n]$, we use the standard notations $\ff(X):=\{F-X: F\in \ff\}$ and $\ff(\bar X):=\{F: F\cap X=\emptyset, F\in \ff\}$.
Let $n>k(s+1)$ and let $\ff\subset {[n]\choose k}$, $\nu(\ff)=s$. The {\it $s$-diversity} $\gamma_s(\ff)$ of $\ff$ is defined as $$\gamma_s(\ff):=\min\Big\{|\ff(\bar R):R\in {[n]\choose s}\Big\}.$$
For $s=1$ we get back to the usual notion of diversity.
Recall also the construction $$\aaa_2:=\aaa_2(n,k,s):=\Big\{A\in {[n]\choose k}:|A\cap[2s+1]|\ge 2\Big\}.$$
\begin{thm}\label{thm1}
If $\ff\subset {[n]\choose k}, \ \nu(\ff)=s$ and $n>n_0(k,s)$, then
\begin{equation}\label{eq1}
  \gamma_s(\ff)\le \gamma_s(\aaa_2)=|\aaa_2\big(\bar{[s]}\big)|.
\end{equation}
\end{thm}
Note that $|\aaa_2\big(\bar{[s]}\big)| = (1+o(1)){s+1\choose 2}{n-s-2\choose k-2}$.

We also note that the motivation for this theorem is the famous Erd\H os Matching Conjecture. If one would have an analogous result in the range of the parameters where the EMC is still open, it would most likely lead to some major progress in the conjecture.
\begin{proof} For a family $\ff$ satisfying the requirements of the theorem,
  let $\tilde\ff\supseteq \ff$ be an {\it extended family} satisfying the following conditions:
  \begin{itemize}
    \item[(i)] For any $\tilde F\in \tilde \ff$ there exists $F\in \ff$, such that $\tilde F\subseteq F$;
    \item[(ii)] $\nu(\tilde\ff)=s$;
    \item[(iii)] $\tilde \ff$ is {\it maximal} with respect to (i), (ii).
  \end{itemize}
We prove \eqref{eq1} by induction on $s$. The case $s=1$ is implied by the result of \cite{Kup21} on the diversity of intersecting $k$-uniform families.

\textbf{Case 0.  $\exists x\in [n]: \{x\}\in \tilde\ff$. } Then $\nu(\tilde\ff(\bar x))=s-1$ and by induction
$$\gamma_s(\ff)\le \gamma_{s-1}(\ff(\bar x))\le \gamma_{s-1}(\aaa_2(n-1,k,s-1))<\gamma_s(\aaa_2(n,k,s)).$$

\textbf{Case 1. } For every $\tilde F\in \tilde\ff$ we have $|\tilde F|\ge 2$.

Define the (edge set of the) graph $\G$ by $\G:=\{\tilde F\in\tilde \ff: |\tilde F|=2\}$. Then \begin{equation}\label{eqasym}|\ff|\le \big(|\G|+o(1)\big){n-2\choose k-2}.\end{equation}
by the argumentation from \cite{Fra78b}.\footnote{
Namely, the minimal sets with respect to containment cannot form any sunflowers of size more than $ks$. Therefore, their number is less than $(ks)^k\cdot k!$ by a classic result of Erd\H os and Rado \cite{ER}.
}
Combined with the following lemma, this concludes the proof of the theorem (which we check below).
\begin{lem}\label{lemdiv2}
  Let $\G$ be (the edge set of) a graph on $n$ vertices with $\nu(\G)=s$, $s\ge 1$ and with no isolated vertices. Then there exists an $s$-element set $R$, such that
  \begin{equation}\label{eq2} |\G(\bar R)|\le {s+1\choose 2},  \end{equation}
  with equality if and only if $\G$ forms a complete graph on $2s+1$ vertices. Moreover, if $\G$ is not a subgraph of a complete graph on $2s+1$ vertices, then
\begin{equation}\label{eq3}  |\G(\bar R)|\le {s\choose 2}+1\end{equation}
\end{lem}
We remark that Lemma~\ref{lemdiv2} completely solves the question of maximizing the $s$-diversity for graphs, and even provides us with a Hilton-Milner-type result.
\begin{proof}

  Let $(x_i,y_i)$, $i\in [s]$, be a maximal matching in $\G$. Let $Z:=[n]\setminus \cup_{i=1}^s\{x_i,y_i\}$ be the remaining vertices.
  Note that $Z$ is an independent set. Note also that $\nu(\G)=s$ implies that there are  no $z,z'\in Z$ and $i\in[s]$ such that $(z,x_i),(z',y_i)\in \G$. That is, for $i\in[s]$, either:\begin{itemize}
  \item[(a)] there is a unique $z_i\in Z$ with $(z_i,x_i), (z_i,y_i)\in \G$ and none of $(z,x_i), (z,y_i)$ are in $\G$ for $z\in Z\setminus \{z_i\}$, or
  \item[(b)] we can suppose by symmetry that $(z,y_i)\notin \G$ holds for all $z\in Z$.
\end{itemize}

  Let us delete the set $X:=\{x_1,\ldots, x_s\}$. The only possible edges that remain in $\G(\bar X)$ are edges $(y_i,y_j)$, which account for at most ${t\choose 2}$ edges, and for each $1\le i\le t$ at most one edge of the form $(y_i, z_i)$ (note that $z_i=z_{i'}$ is possible for $i\ne i'$!). Thus, there are at most ${s\choose 2}+s = {s+1\choose 2}$ edges.

  If $z_i$ is the same for all $i$ for which it is defined, then $\G$ is a subgraph of a complete graph on $2s+1$ vertices. If for some $i\ne j\in [s]$ $z_i,z_j$ are different then there is no edge between $y_i$ and $y_j$. Indeed, otherwise we can replace $x_iy_i$ and $x_jy_j$ with $z_ix_i,$ $y_iy_j$ and $x_jz_j$, thus obtaining a matching of size $s+1$. Moreover, if not all $z_i$ are the same, and $z_i$'s are defined for $t$ indices $i\in [s]$, then there are at least $t-1$ pairs $\{i,j\}$ for which $z_i\ne z_j$. In this case, the total number of edges is at most ${s\choose 2}+t-(t-1)={s\choose 2}+1$. \end{proof}

Lemma~\ref{lemdiv2} and \eqref{eqasym} imply that $\gamma_s(\ff)\le \big({s+1\choose 2}+o(1)\big){n-2\choose k-2}$, moreover, ${s+1\choose 2}$ can be replaced by ${s+1\choose 2}-1$ provided  $\mathcal G$ is not a complete graph on $2s+1$ vertices. Since $|\mathcal A_2(\bar[s])|=(1+o(1)){s+1\choose 2}{n-2\choose k-2}$, we conclude that the family $\mathcal G(\bar R)$ from \ref{eqasym} must have size ${s+1\choose 2}$, and thus $\mathcal G$ must be a complete graph on $2s+1$ vertices, say $[2s+1]$. This clearly forces all the sets from $\ff$ to intersect $[2s+1]$ in at least $2$ elements. This concludes the proof of Theorem~\ref{thm1}.\end{proof}

\section{Non-uniform families}\label{sec7}
We say that $\ff\subset 2^{[n]}$ is {\it $r$-wise $t$-union} if for any $F_1,\ldots, F_r\in \ff$ we have $|F_1\cup\ldots \cup F_r|\le n-t$. Daykin~\cite{Day2} proposed the following problem: given an $r$-wise $t$-union family $\ff\subset 2^{[n]}$,  what is the maximum possible minimum degree of $\ff$?
\begin{obs}
  $\ff\subset 2^{[n]}$ is an $r$-wise $t$-intersecting family with diversity $\gamma(\ff)=d$ if and only if the family of the complements $\ff^c:=\{[n]\setminus F:F\in \ff\}$ is an $r$-wise $t$-union family of minimum degree $d$.
\end{obs}
Daykin \cite{Day2} proposed the following conjecture.
\begin{gypo}[\cite{Day2}]\label{conjday}
  If $r\ge 3$ and $\ff\subset 2^{[n]}$ is $r$-wise union then its minimum degree is at most $2^{n-r-1}$.
\end{gypo}
This was proved in \cite{DF2} for $r\ge 25$ and in \cite{F91} for $r\ge 5$.
For the corresponding values, we have
\begin{equation}\label{eqfr1}
  \gamma(\ff)\le 2^{n-r-1} \ \ \ \text{ for all $r$-wise intersecting families } \ \ \ff\subset 2^{[n]}, \text{ where } \ r\ge 5.
\end{equation}
In \cite{F91}, it was even shown that
\begin{small}\begin{equation}\label{eqfr2}
  \gamma(\ff)\le 2^{n-r-t} \ \ \ \text{for all $r$-wise $t$-intersecting families} \ \ \ff\subset 2^{[n]}, \ \text{ where } \ r\ge 5\ \text{ and } \ t\le 2^r-2r.
\end{equation}\end{small}
The construction showing that \eqref{eqfr1} and \eqref{eqfr2} are optimal is
$$\big\{A\subset [n]: |A\cap [r+t]|\ge r+t-1\big\}.$$
Note that removing from the family above some or all of the members containing $[r+t]$ will not alter the diversity. Therefore, the optimal families are not unique.

Next, let us consider the case of ($2$-wise) $t$-intersecting families for $t\ge 2$. The following classic result was proved by Katona \cite{Kat}.
\begin{thm}[\cite{Kat}]\label{thmkat} If $\ff\subset 2^{[n]}$ is $t$-intersecting then
$$|\ff|\le \begin{cases}
             \sum_{i=(n+t)/2}^n{n\choose i} & \mbox{if } n+t \text{ is even}, \\
             2\sum_{i=(n+t-1)/2}^{n-1}{n-1\choose i} & \mbox{if } n+t \text{ is odd}.
           \end{cases}$$
\end{thm}
Both functions on the right hand side are asymptotically $\big(1-\Theta(n^{-0.5})\big)2^{n-1}$ for fixed $t$. The theorem is sharp due to the following examples:
\begin{align*}
  n+t\text{ is even: } &\ \big\{X\subset [n]: |X|\ge (n+t)/2\big\}, \\
  n+t\text{ is odd: } &\ \big\{X\subset [n]: |X\cap [2,n]|\ge (n+t-1)/2\big\}.
\end{align*}

These two families have diversity $ (1- O_t(n^{-0.5}))2^{n-2}$. On the other hand, for any $t$-intersecting $\ff\subset 2^{[n]}$ and $x\in[n]$, $\ff(\bar x)$ is $t$-intersecting on $[n]\setminus \{x\}$. Thus, we may apply Theorem~\ref{thmkat} and conclude that $\gamma(\ff)\le \big(1-\Omega_t(n^{-0.5})\big)2^{n-2}$ for any $t$-intersecting family $\ff\subset 2^{[n]}$.
We formulate these findings in a proposition.
\begin{prop} For integer $t\ge 2$, the largest diversity $\gamma$ of a $t$-intersecting family $\ff\subset 2^{[n]}$ satisfies $$\gamma = \Big(1-\Theta\big(n^{-0.5}\big)\Big)2^{n-2}.$$
\end{prop}

The case $k=2$, $t=1$ is special. In \cite{DF}, the authors managed to prove
$\gamma(\ff)\le \big(1-n^{-1}\big)2^{n-2}$ for any intersecting family $\ff\subset 2^{[n]}$ and constructed an example with of an intersecting family
$\gamma(\ff)\ge \big(1-n^{-0.651})\big)2^{n-2}$. The authors of \cite{DF} conjectured that the construction attaining the lower bound is in fact optimal. Unaware of \cite{DF}, Huang \cite{HH} asked, what is the largest diversity of an intersecting family in $2^{[n]}$. He conjectured that it is at most $\big(1-\Omega(n^{-0.5}))\big)2^{n-2}$. In what follows, we solve the problem asymptotically using the tools coming from the Analysis of Boolean Functions. Both conjectures have negative answer.

\begin{prop} The largest diversity $\gamma$ of an intersecting family $\ff\subset 2^{[n]}$ satisfies $$\gamma = \Big(1-\Theta\big(\frac{\log n}n\big)\Big)2^{n-2}.$$
\end{prop}

\begin{proof}
Consider the uniform measure $\mu$ on all subsets of $2^{[n]}$. The {\it influence} $I_i(\ff)$ of coordinate $i$ in $\ff$ is $$I_i(\ff):=2\mu\big(\{F\subset [n]: i\notin F, |\{F, F\cup\{i\}|\cap \ff=1\}\big),$$ where here and below $\mu$ stands for the uniform measure on $2^{[n]}$. The {\it total influence} is $I(\ff):=\sum_i I_i(\ff)$. In case if $\ff$ is closed upwards, we have \begin{equation}\label{eqnon1} I_i(\ff) = 2\mu\big(\{F\in \ff:i\in F\}\big)-2\mu\big(\{F\in \ff:i\notin F\}\big).\end{equation}
By the definition of diversity, it is easy to see that
\begin{equation}\label{eqnon2}
\frac 12 \max_{i\in[n]} I_i(\ff)+2\gamma(\ff) = \mu(\ff).
\end{equation}
Since increasing the size of the family cannot reduce the diversity, we may restrict ourselves to the families having measure $1/2$. Therefore, to maximize diversity, one has to minimize maximal influence of one coordinate in a family. By the famous KKL theorem \cite{KKL}, we have $\max_{i\in[n]} I_i(\ff)=\Omega(\frac{\log n}n)$, and thus $\gamma(\ff)=
\big(\frac 14-\Omega(\frac{\log n}n)\big)2^n$.

For the lower bound, consider the following example. Arrange the elements of $[n]$ on the circle, and for each set $S\subset 2^{[n]}$ form a sequence $\mathbf{u}:=(u_1,u_2,\ldots)$, where $u_i$ is the length of the $i$-th longest run of consecutive $1$'s; similarly, $\mathbf{z}:=(z_1,z_2,\ldots )$ is the sequence, in which $z_i$ is the $i$-th longest run of consecutive $0$'s.  Form an intersecting family by including all sets, for which its sequence $\mathbf{u}$ is lexicographically bigger than $\mathbf{z}$. If $n$ is even, then there may be sets which have $\mathbf u = \mathbf z$, and we take a half of them, one out of each pair of two opposite sets, arbitrarily (note that such sets must have exactly $n/2$ ones). This example was proposed by Gil Kalai \cite{GKM}, where the properties of this family were briefly described. In \cite{Kup21}, the first author gave a detailed proof of the fact that each coordinate has influence $O(\log n/n)$ in such a family for odd $n$. We note that the analysis for the even case can be done along the same lines, just taking into account that there are very few sets for which the sequences $\mathbf u$ and $\mathbf z$ coincide. Therefore, we conclude that the minimum of the largest individual influence in an intersecting monotone family is $\Theta(\log n/n)$, and, substituting into \eqref{eqnon2}, we get that the diversity of this family is  $\big(\frac 14-O(\frac{\log n}n)\big)2^n$.
\end{proof}

\section{Open problems}\label{sec8}

There are many interesting open problems that involve diversity-related quantities. We have already mentioned some of them, in particular, Conjecture~\ref{conjday}.

The next problem was discussed in Section~\ref{sec6}.
\begin{prb}Extend Theorem~\ref{thm1} to the range $n\ge Cks$ for some absolute $C>0$.\end{prb}

In Theorem~\ref{thmstabkk} we defined the diversity of families in terms of their distance from the family of all subsets of an $n$-element set. One may generalize this notion and, for a family $\ff\subset {[n]\choose k}$, define its {\it diversity with respect to a initial segment in the colex order}: $\gamma_{colex}(\ff,t):=\min|\ff\setminus \mathcal C(n,k,t)|$, where the minimum is taken over all possible orderings of the ground set. It seems plausible that the following extension of our stability result for the Kruskal--Katona is possible. We note that it generalizes the results of F\"uredi and Griggs \cite{FG}, who characterized the cases in which equality holds in the Kruskal--Katona theorem.

\begin{prb}  Let $s>0$ and $n_s>n_{s-1}>\ldots>n_1> n>k>0$ be integers,  and assume that $\ff$ is a family of $k+s$-element sets, such that $|\ff|\ge t:={n_s\choose k+s}+\ldots+{n_1\choose k+1}+{n\choose k}-{y\choose n-k}+{x\choose k-1}$ and $\gamma_{colex}(\ff,t)\ge {x\choose k-1}$ for some real numbers $k-1\le x\le n-3$, $n-k\le y\le n-3$. Then $\partial(\ff) \ge {n_s\choose k+s-1}+\ldots+{n_1\choose k}+{n\choose k-1}-{y\choose n-k+1}+{x\choose k-2}$.
\end{prb}

Fix integers $0\le a<b\le k$. Find $\max\min|\ff(\bar B)|$, where the maximum is over all $\ff\subset {[n]\choose k}$, such that $\ff$ is intersecting, $\tau(\ff)=b$ and the minimum is over all $B\in {[n]\choose a}$.

Probably, the most interesting case is $b=k$, that is, of intersecting $k$-graphs $\ff$ that cannot be covered by $k-1$ vertices. A classical result of Erd\H os and Lov\'asz \cite{EL} shows that $|\ff|\le k^k$, independent of $n$ (cf. \cite{Fra28} for the current best bound).

In \cite{FOT2} constructions with $|\ff|>\Big(\frac k2\Big)^k$ were given. However, almost all edges contain one specific vertex. Slightly more precisely, the number of edges not containing that vertex is polynomial in $k$.

Let $m(k,a)$ be the maximum above over all $n$ for $b=k$. That is, the value of $m(k,0)$ is the answer to the Erd\H os--Lov\'asz problem.

\begin{gypo}
  $m(k,1)/m(k,0)\to 0$ as $k\to \infty$.
\end{gypo}

\begin{gypo}
  There is an absolute constant $c>1$ such that $m(k,k-1)>c^k$ for all $k\ge 3$.
\end{gypo}
It is easy to show that $m(3,2)=2$. Using product-type constructions, we managed to show that there is no such $\epsilon>0$ such that $m(k,k-1)<2^{k^{1-\epsilon}}$ holds for all $k$.


\begin{thebibliography}{100}
\bibitem{Day} D. E. Daykin, {\it A Simple Proof of the Kruskal–Katona Theorem}, J. Combin. Theory Ser. A 17
(1974), 252--253.

\bibitem{Day3} D.E. Daykin, {\it Erd\H os-Ko-Rado from Kruskal-Katona}, Journal of Combinatorial Theory Ser. A 17 (1974), N2, 254--255.


\bibitem{Day2} D.E. Daykin, {\it Minimum subcover of a finite set} Amer. Math. Monthly, 85 (1978), 766.

\bibitem{DF2} D.E. Daykin and P. Frankl, {\it Sets of finite sets satisfying union conditions}, Mathematika 29 (1982), N1, 128--134.


\bibitem{DF} I. Dinur, E. Friedgut, {\it Intersecting families are essentially contained in juntas}, Combinatorics, Probability and Computing 18 (2009), 107--122.

\bibitem{EKR} P. Erd\H os, C. Ko, R. Rado, \textit{Intersection theorems for systems of finite sets}, The Quarterly Journal of Mathematics, 12 (1961) N1, 313--320.

\bibitem{EL} P. Erd\H os, L. Lov\'asz, {\it Problems and results on 3-chromatic hypergraphs
and some related questions}, in: Infnite and Finite Sets, Proc. Colloq.
Math. Soc. J\'anos Bolyai, Keszthely, Hungary (1973), North-Holland,
Amsterdam (1974), 609--627.

\bibitem{ER} P. Erd\H os and R. Rado {\it Intersection theorems for systems of
sets}, J. London Math. Soc. 35 (1960), 85--90.

\bibitem{Feg} C. Feghali, {\it Multiply Erd\H {o} s-Ko-Rado Theorem.} arXiv:1712.09942.

\bibitem{Fra78b} P. Frankl, {\it On intersecting families of finite sets,} Journal of Combinatorial Theory Ser. A 24 (1978), N2, 146--161.

\bibitem{Fra78} P. Frankl, {\it The Erdos-Ko-Rado theorem is true for $n= ckt$}, Combinatorics (Proc. Fifth Hungarian Colloq., Keszthely, 1976). Vol. 1, 1978.


\bibitem{Fra1} P. Frankl,  \textit{Erdos-Ko-Rado theorem with conditions on the maximal degree}, Journal of Combinatorial Theory, Series A 46 (1987), N2, 252--263.

\bibitem{Fra3} P. Frankl,  \textit{The shifting technique in extremal set theory}, Surveys in combinatorics 123 (1987), 81--110.

\bibitem{F91} P. Frankl, {\it Multiply-intersecting families}, Journal of Combinatorial Theory, Series B 53 (1991), N2, 195--234.

\bibitem{Fra6} P. Frankl, {\it Antichains of fixed diameter}, Moscow Journal of Combinatorics and Number Theory 7 (2017), N3

\bibitem{Fra28} P. Frankl, {\it A near exponential improvement on a bound of Erd\H os and Lov\' asz}, preprint, \url{https://users.renyi.hu/~pfrankl/2017-4.pdf}

\bibitem{FK5} P. Frankl, A. Kupavskii, {\it Counting intersecting and pairs of cross-intersecting families, Comb. Probab. Comput.} 27 (2018), N1, 60--68.

\bibitem{FLST} P. Frankl, S.J. Lee, M. Siggers, and N. Tokushige, {\it An Erd\H os-Ko-Rado theorem for cross $t$-intersecting families}, Journal of Combinatorial Theory, Series A 128 (2014), 207--249.

\bibitem{FOT2} P. Frankl, K. Ota and N. Tokushige, {\it Covers in uniform intersecting families and a counterexample
to a conjecture of Lov\'asz}, J. Comb. Theory Ser. A 74 (1996), 33--42.

\bibitem{FT} P. Frankl and N. Tokushige, {\it Some best possible inequalities concerning cross-intersecting families}, Journal of Combinatorial Theory, Series A 61 (1992), N1, 87-97.

\bibitem{FG} Z. F\"uredi and J. R. Griggs, {\it Families of finite sets with minimum shadows,} Combinatorica 6 (1986), N4, 355--363.

\bibitem{Hil} A.J.W. Hilton, {\it The Erd\H os-Ko-Rado theorem with valency conditions} (1976), unpublished manuscript.
\bibitem{HH} H. Huang, {\it Two extremal problems on intersecting families}, Eur. J. Comb. 76 (2019), 1-9.
\bibitem{IK} F. Ihringer, A. Kupavskii, {\it Regular intersecting families}, arXiv:1709.10462

\bibitem{KKL} J. Kahn, G. Kalai, N. Linial, {\it The influence of variables on Boolean functions}, Proc. 29th Annual Symposium on Foundations of Computer Science (1988),  68--80.

\bibitem{GKM} G. Kalai, A post on MathOverflow, \url{https://mathoverflow.net/questions/105086/kahn-kalai-linial-for-intersecting-upsets}

\bibitem{Kat} G.O.H. Katona, {\it Intersection theorems for systems of finite sets}, Acta Math. Acad. Sci. Hungar. 15 (1964), 329--337.

\bibitem{Ka} G.O.H. Katona, {\it A theorem of finite sets}, Theory of Graphs, Proc. Coll. Tihany 1966, Akad,
Kiado, Budapest, 1968; Classic Papers in Combinatorics (1987), 381--401.

\bibitem{Kee} P. Keevash, {\it Shadows and intersections: Stability and new proofs,} Advances in Mathematics 218 (2008), N5, 1685--1703.

\bibitem{KLo} P. Keevash and E. Long, {\it Stability for vertex isoperimetry in the cube} (2018) arXiv:1807.09618


\bibitem{Kr} J.B. Kruskal, {\it The Number of Simplices in a Complex}, Mathematical optimization techniques
251 (1963), 251--278.
\bibitem{KZ} A. Kupavskii, D. Zakharov, {\it Regular bipartite graphs and intersecting families}, 	J. Comb. Theory Ser. A 155 (2018), 180-189.

\bibitem{Kup21} A. Kupavskii, {\it Diversity of intersecting families}, Eur. J. Comb. 74 (2018), 39-47.

\bibitem{Kup22} A. Kupavskii,  {\it Structure and properties of large intersecting families}, 	arXiv:1710.02440

\bibitem{LP} N. Lemons, C. Palmer, \textit{Unbalance of set systems}, Graphs and Combinatorics 24 (2008), N4, 361--365.

\bibitem{MT} M. Matsumoto and N. Tokushige, {\it The exact bound in the Erd\H os-Ko-Rado theorem for cross-intersecting families}, Journal of Combinatorial Theory, Series A 52 (1989), N1, 90-97.

\bibitem{ODW} R. O'Donnell and K. Wimmer, {\it KKL, Kruskal–Katona and monotone nets}, SIAM Journal on Computing 42 (2013), N6, 2375--2399.

\bibitem{Wil} R.M. Wilson, {\it The exact bound in the Erd\H os-Ko-Rado theorem,} Combinatorica 4 (1984), N2, 247--257.

\end{thebibliography}
\end{document}